\documentclass[11pt, a4paper]{amsart}
\usepackage{graphicx}
\usepackage{epstopdf}
\usepackage[permil]{overpic}
\usepackage{tikz}
\usetikzlibrary{shapes,arrows}

\usepackage{amsmath}
\usepackage{amsthm}
\usepackage{amsfonts}
\usepackage{enumerate}
\usepackage{amssymb}

\usepackage{etoolbox}
\preto\align{\par\nobreak\small\noindent}
\expandafter\preto\csname align*\endcsname{\par\nobreak\small\noindent}

\newcommand{\C}{\mathbb{C}}						
\newcommand{\R}{\mathbb{R}}						

\newcommand{\Hinfty}{H^\infty}														
\newcommand{\stable}{A}	
\newcommand{\pstable}{\mathbf{P}}													
\newcommand{\fractions}[1]{Q(#1)}	

\newcommand{\error}{e}													
\newcommand{\measurement}{y}										
\newcommand{\reference}{\measurement_r}												
\newcommand{\disturbance}{d}										

\newcommand{\Generator}{\Theta}											
\newcommand{\generator}{\theta}											
\newcommand{\Plant}{p}		
\newcommand{\Controller}{c}													
\newcommand{\Closedloop}[1]{H( #1 )}			

\newtheorem{theorem}{Theorem}[section]
\newtheorem{lemma}[theorem]{Lemma}
\newtheorem{corollary}[theorem]{Corollary}
\newtheorem{remark}[theorem]{Remark}

\theoremstyle{definition}
\newtheorem{definition}[theorem]{Definition}
\newtheorem{example}[theorem]{Example}

\newcommand*{\QEDexample}{\hfill\ensuremath{\blacksquare}}%
\begin{document}

\pagestyle{empty} 


\title{
A fractional representation approach to the robust regulation problem \\  for SISO systems}
\thispagestyle{plain}
\author{P. Laakkonen$^\dagger$ and A. Quadrat$^*$}
\address{$^\dagger$ Laboratory of Mathematics, Tampere University of Technology, PO.\ Box 553, 33101 Tampere, Finland}
\address{$^*$ Inria Lille - Nord Europe, Non-A project, Parc Scientifique de la Haute Borne, 40 Avenue Halley, Bat. A - Park Plaza, 59650 Villeneuve d'Ascq, France.}

\begin{abstract}
The purpose of this article is to develop a new approach to the robust regulation problem for plants which do not necessarily admit coprime factorizations. The approach is purely algebraic and allows us dealing with a very general class of systems in a unique simple framework. We formulate the famous internal model principle in a form suitable for plants defined by fractional representations which are not necessarily coprime factorizations. By using the internal model principle, we are able to give necessary and sufficient solvability conditions for the robust regulation problem and to parameterize all robustly regulating controllers.
\end{abstract}


\maketitle

\section{Introduction}

Robustness of controllers is of fundamental importance since it allows them to work under uncertain conditions. Regulating controllers can asymptotically track a given reference signal. Robustness means that the controller remains regulating despite small perturbations of the system. For example, modeling errors, model simplifications and attrition of components in a real world application can be seen as perturbations of the system. The robust regulation problem is to find a robustly regulating controller.

Robust regulation of finite-dimensional plants is well-understood \cite{Davison1976a,FrancisWonham1975a,Vidyasagar}. The finite-dimensional theory has been generalized to infinite-dimensional plants and signals by several authors. See, for instance, \cite{CallierDesoer1980, HamalainenPohjolainen2000, HamalainenPohjolainen2010, Haraetal1988, ImmonenPhd, LaakkonenPhd, NettThesis, PaunonenPohjolainen2012, Ylinenetal2006} and the references therein. One of the most fundamental results of robust regulation is the internal model principle, which states that any robustly regulating controller contains a suitably reduplicated model of the dynamics to be tracked.

In the frequency domain, the robust regulation problem is an algebraic problem. Vidyasagar formulated and solved it by using coprime factorizations over the ring of stable rational transfer functions \cite{Vidyasagar}. Vidyasagar's results state the internal model principle, give a necessary and sufficient solvability condition of the problem, and parameterize all robustly regulating controllers in a remarkably simple form. These results have been generalized to fields of fractions over rings suitable for distributed parameter systems and/or infinite-dimensional reference and disturbance signals \cite{CallierDesoer1980, HamalainenPohjolainen2000, Haraetal1988, LaakkonenPhd, NettThesis, Ylinenetal2006}. The common feature of the results is that they require the existence of coprime factorizations. This is problematic since all plants do not possess coprime factorizations \cite{Anantharam1985, Mori1999}, or their existence is not known \cite{LaakkonenPhd, Mori2002}.

In this paper, we develop robust regulation theory of single-input single-output (SISO) plants based on stabilizability results of \cite{Quadrat2003}. The advantage of the theory presented in \cite{Quadrat2003} is that it uses no coprime factorizations and allows us to develop theory with very few assumptions. We only need to define a commutative ring $\stable$ of stable elements with a unit and having no zero divisors to start with. The plants are just elements in the field of fractions over $\stable$. This makes the theory applicable in several different classes of infinite-dimensional systems, for instance in those of \cite{LaakkonenPhd, Logemann1993}. From the theoretic point of view, the choice of $\stable$ is irrelevant, but when applying the results, the choice of $\stable$ depends naturally on the problem at hand. Examples of rings motivated by systems theoretic applications involve $\Hinfty$ and the Callier-Desoer algebra where all stabilizable plants have coprime factorizations, $\stable:=\R[x^2,x^3]$ of Example \ref{exa:::Example1} with plants without weakly coprime factorizations, and $\pstable$ of \cite{LaakkonenPhd}, for which the existence of (weakly) coprime factorizations of stabilizable transfer functions is not known.

The abstract algebraic approach to robust regulation has received only little attention this far. In the last chapter of his book \cite{Vidyasagar}, Vidyasagar discussed the generalization of finite-dimensional stabilization and regulation theory to infinite-dimensional systems. Unfortunately, the part concerning robust regulation uses coprime factorizations and therefore is not applicable for general rings. The same is true for the theory developed in \cite{NettThesis}. In addition, both of the above references use topological notions in the study of robustness. It is possible to do without by defining the robustly regulating controllers so that they are exactly the ones that are regulating for every plant they stabilize. This definition splits the robust regulation problem into two parts: robust regulation that involves constructing an internal model into the controller and robust stabilization that involves the topological aspects of the problem. In this article, we focus on the former. Robustness of stability is well-understood in many physically interesting algebraic structures \cite{CurtainZwart, Vidyasagar} as well as in the abstract setting \cite{Quadrat2015,Vidyasagaretal1982}.

By using the fractional representation approach, we generalize the theory of \cite{Vidyasagar} to the plants which do not necessarily possess coprime factorizations. The main contributions of this article are:
\begin{itemize}
\item We give a reformulation of the internal model principle without using coprime factorizations.
\item We give a checkable necessary and sufficient condition for solvability of the robust regulation problem.
\item We parameterize all robustly regulating controllers for signal generators with a weakly coprime factorization.
\end{itemize}
The internal model principle and the solvability condition can be found in the preliminary version \cite{LaakkonenQuadrat2015} of this article. However, in this article, we require only weakly coprime factorizations instead of coprime factorizations, which extends some of the results of \cite{LaakkonenQuadrat2015}. 
Theorem \ref{thm:::CoprimeGeneratorSolvability} and Corollary \ref{cor:::AllRobustlyRegulatingControllers}, which give a parametrization of all robustly regulating controllers, are new. We formulate the results of this paper using fractional representations. For fractional ideal approach, see \cite{LaakkonenQuadrat2015}.

The remaining part of the paper is organized as follows. Notations, preliminary results, and the problem formulation are given in Section \ref{sec:::Preliminaries}. The internal model principle is considered in Section \ref{sec:::IMP}. Section \ref{sec:::Solvability} contains solvability considerations and, by using the results of the section, we are able to give a parametrization of all robustly regulating controllers. In Section \ref{sec:::Examples}, we illustrate the theoretical results  by examples. Finally, the concluding remarks are made in Section \ref{sec:::Conclusions}.


\section{
The problem formulation}\label{sec:::Preliminaries}

Let $A$ be an {\em integral domain}, namely a commutative ring with a unit element 1 and without zero divisors \cite{Rotman}. We denote by $\stable^{l \times m}$ the $\stable$-module of $l \times m$ matrices with entries in $\stable$ and by 
\begin{align*}
\fractions{\stable}:=\left \{ \frac{n}{d} \; | \; 0 \neq d, \, n \in \stable \right \}
\end{align*}
{\em the field of fractions} of $\stable$. 


\begin{definition}
\begin{enumerate}
\item An element  $h \in \fractions{\stable}$ (resp., a matrix $H\in \fractions{\stable}^{l \times m}$) is said to be \emph{stable} if we have $h \in \stable$ (resp., $H \in \stable^{l \times m}$) and \emph{unstable} otherwise.

\item A controller $\Controller\in\fractions{\stable}$ \emph{stabilizes} $\Plant\in\fractions{\stable}$ if the closed loop system of Figure \ref{fig:::Closedloop} from $(\reference \quad \disturbance)^T$ 
to 
$(e \quad u)^T$ given by
\begin{align*}
\hspace{-4mm}
\Closedloop{\Plant,\Controller}:=
\left( 
\begin{matrix}
\dfrac{1}{1-\Plant \, \Controller} & \dfrac{\Plant}{1-\Plant \, \Controller} \vspace{1mm}\\
\dfrac{\Controller}{1-\Plant \, \Controller} & \dfrac{1}{1-\Plant \, \Controller} 
\end{matrix} \right)
\end{align*}
is stable, i.e., if we have $\Closedloop{\Plant,\Controller} \in \stable^{2 \times 2}$. 
\end{enumerate}
\end{definition}

\tikzstyle{element}=[draw,fill=blue!20,rectangle,minimum height=3em, minimum width=6em] 
\tikzstyle{terminal}=[circle,draw] 
\tikzstyle{fleche}=[->,thick,rounded corners=4pt] 
\tikzstyle{sum} = [draw, fill=black!20, circle, node distance=1.5cm]

\begin{figure}
\begin{center}
\begin{tikzpicture} 
\node[above right] at (0,0) {$y_r$}; 
\node[above right] at (.8,0) {$+$}; 
\node[above right] at (1.5,-0.6) {$+$}; 
\node[sum] (s) at (1.5,0) {};
\node[above right] at (2,0) {$e$}; 
\node[element] (c) at (4,0) {$c$}; 
\node[above left] at (6,-2) {$u$}; 
\node[above left] at (8,-2) {$d$}; 
\node[above left] at (7.2,-2.5) {$+$}; 
\node[above right] at (6.0,-1.8) {$+$}; 
\node[above left] at (0.5,-2) {$y$}; 

\node[element] (p) at (4,-2) {$p$}; 
\node[sum] (s2) at (6.5,-2) {};

\draw[fleche] (0,0) -- (s);
\draw[fleche] (s) -- (c);
\draw[fleche] (c) -- (8,0);
\draw[fleche] (6.5,0) -- (s2);
\draw[fleche] (s2) -- (p);
\draw[fleche] (8,-2) -- (s2);
\draw[fleche] (p) -- (0,-2);
\draw[fleche] (1.5,-2) -- (s);
\end{tikzpicture}
\end{center}
\caption{The control configuration}
\label{fig:::Closedloop}
\end{figure}
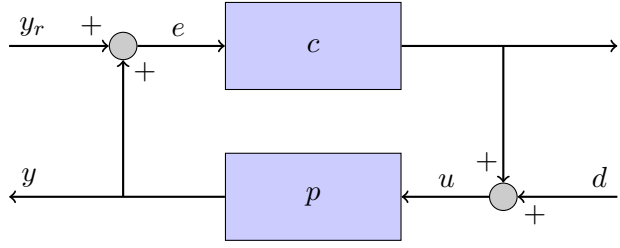


Let ${\rm Stab}(\Plant)$ be the set of all the stabilizing controllers of $\Plant$. Note that $\Controller \in {\rm Stab}(\Plant)$ is equivalent to $\Plant \in {\rm Stab}(\Controller)$.

\begin{definition}
Let $\Generator  \in \fractions{\stable}$. Then, we have:
\begin{enumerate}
\item \emph{A fractional representation} of $\Generator$ is defined by $\Generator=\frac{\gamma}{\generator}$, where  $0\neq \generator,\gamma\in\stable$.


\item A fractional representation $\Generator=\frac{\gamma}{\generator}$ is called \emph{a coprime factorization} if there exist $\alpha, \, \beta\in\stable$ such that $\alpha \, \gamma-\beta \, \generator=1$.

\item A fractional representation $\Generator=\frac{\gamma}{\generator}$ is called \emph{a weakly coprime factorization} if we have:
\begin{align*}
\forall \; k \in \fractions{\stable}: k \, \gamma, \, k \, \generator\in\stable \implies k\in\stable.
\end{align*}
\end{enumerate}
\end{definition}


The approach developed in this article is based on the stabilizability results of \cite{Quadrat2003}. The following theorem combines Theorems~1 and 2 of \cite{Quadrat2003}.
\begin{theorem}\label{thm:::Stability}
The plant $\Plant$ is stabilizable 
if and only if there exist $a, \, b\in\stable$ such that:
\begin{align}\label{eqn:::ab}
\left\{ \begin{array}{l}
a-\Plant \,  b=1, \vspace{1mm} \\
\Plant \, a\in\stable.
\end{array}\right.
\end{align} 
Moreover, a controller $\Controller$ stabilizes $\Plant$ if and only if it is of the form $\Controller=\frac{b}{a}$, where $0\neq a,b\in\stable$ satisfy \eqref{eqn:::ab}. In this case, we have that $a=(1-\Plant \, \Controller)^{-1}$ and $b=\Controller \, (1-\Plant \, \Controller)^{-1}$.

If $0\neq a,b\in\stable$ satisfy \eqref{eqn:::ab}, then all the stabilizing controllers of $\Controller$ are parametrized by
\begin{align}\label{eqn:::AllStabilizingControllers}
\Controller(q_1,q_2):=\frac{b+q_1 \, a^2+q_2 \, b^2}{a+q_1 \, \Plant \, a^2+q_2 \, \Plant \, b^2},
\end{align}
where $q_1, \, q_2 \in \stable$ are such that the denominator of (\ref{eqn:::AllStabilizingControllers}) does not vanish.
\end{theorem}

We make a standing assumption that all the reference and disturbance signals are generated by a fixed signal generator $\Generator\in\fractions{\stable}$, i.e., the reference and disturbance signals are of the form: 
\begin{align*}
\reference:=\Generator \, \measurement_0, \quad \disturbance:=\Generator \, \disturbance_0, \quad \measurement_0, \, \disturbance_0\in\stable.
\end{align*}

\begin{definition}
\begin{enumerate}
\item We say that a controller $\Controller$ is \emph{regulating} $\Plant$ with the signal generator $\Generator$ if
\begin{align*}
\hspace{-6mm}
\error=\left( \begin{matrix} 
\dfrac{1}{1-\Plant \, \Controller} & \dfrac{\Plant}{1-\Plant \, \Controller} \end{matrix} \right) \, \Generator \left(\begin{matrix}\measurement_0 \vspace{1mm} \\ \disturbance_0\end{matrix} \right) \in \stable,
\end{align*}
for all $\measurement_0, \, \disturbance_0\in\stable$, or equivalently if we have:
\begin{align}\label{eqn:::Regulation}
\Generator \, \left( \begin{matrix} 
\dfrac{1}{1-\Plant \, \Controller} & \dfrac{\Plant}{1-\Plant \, \Controller} \end{matrix} \right)   \in \stable^{1 \times 2}.
\end{align}

\item A controller $\Controller$ is called \emph{robustly regulating} with the signal generator $\Generator$ if we have: 
\begin{enumerate}[i.]
\item $\Controller$ stabilizes $\Plant$, i.e., $\Controller \in {\rm Stab}(\Plant)$.
\item $\Controller$ regulates every plant it stabilizes, i.e., if for all $\Plant' \in {\rm Stab}(\Controller)$,  we then have:
\begin{align*}
 \Generator \, \left( \begin{matrix}  \dfrac{1}{1-\Plant' \, \Controller} & \dfrac{\Plant'}{1-\Plant' \, \Controller} \end{matrix} \right) \in \stable^{1 \times 2}.
\end{align*}
\end{enumerate}
\emph{The robust regulation problem} is
the problem of finding a robustly regulating controller. 
\end{enumerate}
\end{definition}

\section{The internal model principle}\label{sec:::IMP}

The first main result of this paper is the formulation of the internal model principle given by the next theorem, namely Theorem~\ref{thm:::IMP}. This result gives a necessary and sufficient condition for a stabilizing controller to be robustly regulating. 

\begin{lemma}\label{lem:::IMP}
A stabilizing controller $\Controller$ is regulating $\Plant$ 
if and only if there exist $\alpha, \, \beta\in\stable$ such that: 
\begin{align}\label{eqn:::RegulationSolvability}
\Generator=\alpha+\beta \, \Controller.
\end{align}
\end{lemma}
\begin{proof}
Let us first assume that $\Controller$ is regulating. 
Then, we have $\alpha = (1-\Plant \, \Controller)^{-1}\, \Generator \in \stable$ and 
$\beta:=-(1-\Plant \, \Controller)^{-1}  \, \Plant \, \Generator \in \stable$, and thus we get:
\begin{align*}
\Generator= \frac{1-\Plant \, \Controller}{1-\Plant \, \Controller}\, \Generator= \frac{\Generator}{1-\Plant \, \Controller}-\frac{\Generator \, \Plant}{1-\Plant \, \Controller} \, \Controller=\alpha+\beta \, \Controller.
\end{align*}

Let us now assume that there exist $\alpha, \, \beta\in\stable$ such that (\ref{eqn:::RegulationSolvability}) holds. Since $\Controller$ stabilizes $\Plant$, we have 
$(1-\Plant \, \Controller)^{-1}, \, \Controller \,  (1-\Plant \, \Controller)^{-1},  \, (1-\Plant \, \Controller)^{-1} \, \Plant \in \stable$, and thus
\begin{align}\label{eq:stabconds}
\left\{\begin{aligned}
\frac{\Generator}{1-\Plant \, \Controller} & = \alpha \, \frac{1}{1-\Plant \, \Controller}+\beta\, \frac{\Controller}{1-\Plant \, \Controller} \in \stable, \vspace{2mm} \\
\frac{\Generator  \, \Plant}{1-\Plant \, \Controller} & = \alpha \, \frac{\Plant}{1-\Plant \, \Controller}
+\beta \, \frac{\Plant \, \Controller}{1-\Plant \, \Controller} \in\stable,
\end{aligned}
\right.
\end{align}
which proves that $\Controller$ is regulating.
\end{proof}


\begin{theorem}\label{thm:::IMP}
A controller $\Controller$ is robustly regulating if and only if it stabilizes $\Plant$ and there exist $\alpha, \beta\in\stable$ such that $\Generator=\alpha+\beta\, \Controller$.
\end{theorem}
\begin{proof}
The necessity can be proved like in Lemma~\ref{lem:::IMP}.
In order to show the sufficiency, we assume that there exist $\alpha,\beta\in\stable$ such that we have $\Generator= \alpha+\beta \, \Controller$. For all $\Plant' \in {\rm Stab}(\Controller)$, the stability of the closed loop $\Closedloop{\Plant',\Controller}$ yields $(1-\Plant' \, \Controller)^{-1}, \, \Controller \,  (1-\Plant' \, \Controller)^{-1}$, $(1-\Plant' \, \Controller)^{-1} \, \Plant' \in \stable$, so we obtain (\ref{eq:stabconds}) where $\Plant$ is replaced by $\Plant'$. 
Thus, $\Controller$ is robustly regulating.
\end{proof}


Theorem~\ref{thm:::IMP} proves that for SISO plants, every stabilizing regulating controller is robustly regulating. This result is well-known in the literature for plants admitting coprime factorizations \cite{Vidyasagar}.

According to Theorem~\ref{thm:::IMP}, we will say that a controller $\Controller$ {\em contains an internal model of the generator} if there exist $\alpha, \, \beta \in \stable$ such that $\Generator=\alpha+\beta \, \Controller$. This means that the instability generated by the signal generator $\Generator$ must be built into a robustly regulating controller $\Controller$.

Next, we ask whether the instability generated by the signal generator $\Generator$ can be represented by a single stable element $\generator$. By this, we mean that a controller $\Controller$ that solves the robust regulation problem with the signal generator $\generator^{-1}$ is also robustly regulating with $\Generator$. The following corollary shows that the denominator $\generator$ of any factorization is such an element. 

\begin{corollary}\label{cor:::IMP2}
Let $\Generator=\frac{\gamma}{\generator}$ be a fractional representation of the signal generator. If $\Controller \in {\rm Stab}(\Plant)$ and there exist $\alpha, \, \beta\in\stable$ such that $\generator \, (\alpha+\beta \, \Controller)=1$,
then $\Controller$ solves the robust regulation problem. 
\end{corollary}
\begin{proof}
If $\Controller \in {\rm Stab}(\Plant)$ and if there exist $\alpha, \, \beta\in\stable$ such that $\generator \, (\alpha+\beta \, \Controller)=1$, then $\theta \neq 0$ and $\generator^{-1}=\alpha+\beta \, \Controller$, which yields $\Generator=\frac{\gamma}{\generator}=(\gamma \, \alpha)+(\gamma \, \beta) \, \Controller$ and proves the result by Theorem~\ref{thm:::IMP}.
\end{proof}

However, $\generator^{-1}$ in Corollary \ref{cor:::IMP2} may not be a ``minimal'' internal model in the sense that a robustly regulating controller with the signal generator $\Generator$ is not necessarily robustly regulating with $\generator^{-1}$. The next theorem shows that the denominator of a weakly coprime factorization is minimal in this sense.

\begin{theorem}\label{thm:::IMP_WeaklyCoprime}
If $\Generator=\frac{\gamma}{\generator}$ is a weakly coprime factorization, then $\Controller$ solves the robust regulation problem if and only if $\Controller \in {\rm Stab}(\Plant)$ and there exist $\alpha, \, \beta\in\stable$ such that $\generator \, (\alpha+\beta \, \Controller)=1$, i.e. $\Controller$ is robustly regulating for the signal generator $\generator^{-1}$.
\end{theorem}
\begin{proof}
By Theorem \ref{thm:::IMP}, $\generator \, (\alpha+\beta \, \Controller)=1$ is equivalent to that $\Controller$ is robustly regulating with $\generator^{-1}$. The sufficiency follows from Corollary~\ref{cor:::IMP2}.
In order to show the necessity, we assume that $\Controller$ is a robustly regulating controller. Since $\Controller$ is stabilizing, Theorem \ref{thm:::Stability} shows that there exist $a, \, b\in\stable$ satisfying \eqref{eqn:::ab} such that $\Controller=\frac{b}{a}$. Since $\Controller$ is regulating
$
\gamma  \, \frac{a}{\generator}=\Generator \, a=\Generator \, (1-\Plant \, \Controller)^{-1}\in\stable$
and
$\generator \, \frac{a}{\generator}=a\in\stable$.
Weak coprimeness of the factorization $\Generator=\frac{\gamma}{\generator}$ implies that $\frac{a}{\generator}\in\stable$. Similarly, we can show that $\frac{a \, \Plant}{\generator}\in\stable$. By \eqref{eqn:::ab}, we get
\begin{align*}
\frac{1}{\generator}=\dfrac{a}{\generator}-\dfrac{a \, \Plant}{\generator} \, \Controller
\end{align*}
which completes the proof.
\end{proof}


We end this section by showing that a robustly regulating controller of a plant admitting a coprime factorization (e.g., $\Plant \in \stable$) necessarily contains the denominator of a fractional representation of the generator as an internal model.

\begin{theorem}\label{thm:Pcoprime}
Let $\Plant$ admit a coprime factorization $\Plant=\frac{n}{d}$ and $\Controller$ stabilize $\Plant$. Then, $\Controller$ is robustly regulating if and only if the generator $\Generator$ admits a fractional representation $\Generator=\frac{z}{x}$, where $x$ is the denominator of a coprime factorization $\Controller=\frac{y}{x}$. In particular, we have $x \, (\alpha+\beta \, \Controller)=1$ for some $\alpha, \, \beta \in \stable$. Finally, if $\Generator$ admits a coprime factorization $\Generator=\frac{\gamma}{\generator}$, then $x=\delta \, \generator$ for a certain $\delta \in \stable$.
\end{theorem}
\begin{proof}
Let us suppose that $\Controller$ robustly regulates $\Plant$. If $\Plant=\frac{n}{d}$ and $\Controller=\frac{y}{x}$ are coprime factorizations, then a standard result asserts that $\Controller$ stabilizes $\Plant$ if and only if $d \, x-n \, y=u$, where $u$ is an invertible element of $\stable$, i.e. $u^{-1}\in\stable$ \cite{Vidyasagar}. Then, we have:
\begin{align*}
\left \{\begin{aligned}
\frac{\Generator}{1-\Plant \, \Controller}  & = u^{-1}\,d\, x \, \Generator \in \stable, \\
\frac{\Plant \, \Generator}{1-\Plant \, \Controller} & = u^{-1}\,n \, x \, \Generator \in \stable.
\end{aligned}
\right. 
\end{align*}
Therefore, we get
\begin{align*}
x \, \Generator=x \, (u^{-1}\,d\, x \, \Generator)- y \, ( u^{-1}\, n \, x \, \Generator) \in \stable,
\end{align*}
and thus there exists $z \in \stable$ such that $\Generator=\frac{z}{x}$.  Moreover, we have: 
\begin{align}\label{eq:coprime}
x \, (u^{-1}\,  d-u^{-1}\, n \, \Controller)=1.
\end{align}
 
Conversely, if $\Generator=\frac{z}{x}$, where $x$ is the denominator of a coprime factorization $\Controller=\frac{y}{x}$ and $z\in\stable$, then we have $d \, x-n \, y=u$, where $u$ is a unit of $\stable$, which yields (\ref{eq:coprime}) and proves that $\Controller$ robustly regulates $\Plant$ by Corollary~\ref{cor:::IMP2}.

Finally, if $\Generator=\frac{\gamma}{\generator}$ is a coprime factorization, then there exist $\varepsilon, \, \nu \in \stable$
such that $\generator \, \nu-\gamma \, \varepsilon=1$. Then we have $\Generator=\frac{\gamma}{\generator}=\frac{z}{x}$, i.e., $x=\frac{z}{\gamma} \, \generator$, and
\begin{align*}
\delta:=\frac{z}{\gamma}=\frac{z \, (\generator \, \nu-\gamma \, \varepsilon)}{\gamma}=x \, \nu-z \, \varepsilon \in \stable.
\end{align*}
\end{proof}

\section{Solvability of the robust regulation problem}\label{sec:::Solvability}


In this section, we give necessary and sufficient conditions for the solvability of the robust regulation problem. The first lemma gives a solvability condition for stable plants.

\begin{lemma}\label{lem:::StablePlantSolvability}
If $\Plant\in\stable$, then the robust regulation problem is solvable if and only if:   
\begin{align}\label{eqn:::Controller1}
\exists \; \alpha, \, \beta\in\stable: \quad \alpha \, \Generator^{-1}-\beta \, \Plant=1.
\end{align}
\end{lemma}
\begin{proof}
Let us first assume that $\Controller$ is a robustly regulating controller. Theorem~\ref{thm:::Stability} shows that $\Controller=\frac{b}{a}$, where $a,b\in\stable$ satisfy \eqref{eqn:::ab}. Since $\Controller$ is regulating, we have $a \, \Generator 
\in\stable$. Thus, 
$1=a-b \, \Plant=(a \, \Generator) \, \Generator^{-1}-b \, \Plant$,
which proves the necessity.

Let us now assume that there exist $\alpha, \, \beta\in\stable$ such that we have (\ref{eqn:::Controller1}).
If $\alpha=0$, then
\begin{align*}
h \, \Generator^{-1}-(1-h \, \Generator^{-1})\, \beta \, \Plant=1,
\end{align*}
where $h\in\stable \setminus \{0\}$ is chosen so that $h \, \Generator^{-1}\in\stable$. Thus, without restricting generality, we can assume that $\alpha\neq 0$. Since $\beta \, \Plant\in\stable$, we see that $\alpha \, \Generator^{-1}\in\stable$, and $\Plant \, \alpha \, \Generator^{-1}\in\stable$. Thus, the equation \eqref{eqn:::Controller1} implies that $\Controller:=\frac{\beta}{\alpha} \, \Generator$ stabilizes $\Plant$ by Theorem~\ref{thm:::Stability}. Furthermore,
$\Generator \, (1-\Plant \, \Controller)^{-1} =\Generator \, \alpha \, \Generator^{-1}=\alpha\in\stable$,
which is enough to show that $\Controller$ is robustly regulating.
\end{proof}

We now state the main results of this section: two necessary and sufficient solvability conditions for the robust regulation problem. In the next theorem, we convert the problem of solvability into a robust regulation problem of a stable plant. A checkable condition for the solvability follows (see Corollary~\ref{cor:::Solvability}). Let us first state a useful lemma.

\begin{lemma}\label{lemma:product}
Let $\Controller \in {\rm Stab}(\Plant)$, $a:=(1-\Plant \, \Controller)^{-1} \in \stable$, $b:=\Controller\, (1- \Plant \, \Controller)^{-1} \in \stable$ and $\Controller_i \in {\rm Stab}(b \, \Plant)$. Then, we have: 
\begin{align}\label{eqn:::RobustlyRegulatingController}
\Controller_r:=\Controller \, (1+\Controller_i) \in {\rm Stab}(\Plant).
\end{align}
Hence, the controllers of the form
\begin{align}\label{def:subparametrization}
\widetilde{\Controller}(\widetilde{q})=\Controller \, \left( 1+ \frac{\widetilde{q}}{1+b \, \Plant \, \widetilde{q}} \right),
\end{align}
where $\widetilde{q}\in\stable$, stabilize $\Plant$, i.e., $\widetilde{\Controller}(\widetilde{q}) \in {\rm Stab}(\Plant)$ for all $\widetilde{q} \in \stable$. The controllers of the form (\ref{def:subparametrization}) are obtained by choosing $q_1=b \, \widetilde{q}$ and $q_2=-(a \, \Plant) \, \widetilde{q}$ in
(\ref{eqn:::AllStabilizingControllers}), 
and we have:
\begin{align}\label{def:aproduct}
\frac{1}{1-\Plant \, \widetilde{\Controller}(\widetilde{q})}:=a \, (1+b \, p \, \widetilde{q}).
\end{align}
 Finally, if $\Controller_i$ robustly regulates $b\, \Plant$, then $\Controller_r$ is robustly regulating for $\Plant$.
\end{lemma}
\begin{proof}
We clearly have:
\begin{align}\label{eqn:::Help}
\frac{1}{1-\Plant \, \Controller_r}=\frac{1}{(1-\Plant \, \Controller) \, (1-b \, \Plant \, \Controller_i)}.
\end{align}
Moreover, we also have:
\begin{align}
\frac{\Controller_r}{1-\Plant \, \Controller_r} & =\frac{\Controller}{(1-\Plant \, \Controller)} \, \frac{(1+c_i)}{(1-b \, \Plant \, \Controller_i)}, \vspace{1mm} \notag\\ 
\frac{\Plant}{1-\Plant \, \Controller_r} & =\frac{\Plant}{(1-\Plant \, \Controller)} \, \frac{1}{(1-b \, \Plant \, \Controller_i)}.\label{eqn:::Help2}
\end{align}
Now, using $\Controller \in {\rm Stab}(\Plant)$ and $\Controller_i \in {\rm Stab}(b \, \Plant)$, we obtain $\Controller_r \in {\rm Stab}(\Plant)$. Since $b \, \Plant\in \stable$, considering $a'=1$ and $b'=0$, we get $a'- b' \, (b \, \Plant)=1$ and using (\ref{eqn:::AllStabilizingControllers}), all the stabilizing controllers of $b \, \Plant$ are of the form $\frac{q}{1+b \, \Plant \, q}$ for all $q \in \stable$, which shows that $\widetilde{c}(\widetilde{q})$ of (\ref{def:subparametrization}) stabilizes $\Plant$.

By (\ref{eqn:::AllStabilizingControllers}), all the stabilizing controllers of $\Plant$ are 
\begin{align*}
\Controller(q):=\frac{b+q}{a+\Plant  \, q},
\end{align*}
where $q:=q_1 \, a^2+q_2 \, b^2$ and $q_1, \, q_2 \in \stable$. Using (\ref{eqn:::ab}), we then have:
\begin{align*}
\Controller(q) 
=\Controller \, \frac{a \, (b+q)}{b \, (a+\Plant  \, q)}
=c\, \left(1+\frac{q}{b \, (a+\Plant  \, q)} \right).
\end{align*}
Considering $q_1=b \, \widetilde{q}$ and $q_2=-(a \, \Plant) \, \widetilde{q}$ for $\widetilde{q} \in \stable$, we get $q=q_1 \, a^2+q_2 \, b^2=a \, b \, \widetilde{q} \, (a-b \, \Plant)=a \, b \, \widetilde{q}$ and:
\begin{align*}
c(q)=c \, \left(1+ \frac{a \, b \, \widetilde{q}}{a \, b +a \, b^2 \, \Plant \,  \widetilde{q}} \right)=c \, \left(1+ \frac{\widetilde{q}}{1 +b \, \Plant \, \widetilde{q}} \right).
\end{align*}
Substituting $q=a \, b \, \widetilde{q}$ into $a+p \, q$, we get (\ref{def:aproduct}).

If it is assumed that $\Controller_i$ robustly regulates $b\, \Plant$, then $\Generator\, (1-\Plant\, b \, \Controller_i)^{-1}\in\stable$. Thus, (\ref{eqn:::Help}) and (\ref{eqn:::Help2}) both multiplied by $\Generator$ are stable, and $\Controller_r$ is robustly regulating.
\end{proof}

\begin{theorem}\label{thm:::Solvability}
The robust regulation problem is solvable if and only if there exists a stabilizing controller $\Controller=\frac{b}{a}$ such that \eqref{eqn:::ab} holds and there exist $\alpha, \, \beta\in\stable$ such that: 
\begin{align}\label{eq:fund}
\alpha \, \Generator^{-1}-\beta \, b \, \Plant=1.
\end{align}
\end{theorem}
\begin{proof}
If $\Controller=\frac{b}{a}$, with $a$ and $b$ satisfying (\ref{eqn:::ab}), is robustly regulating, then we have $a \, \Generator \in \stable$ and $1=a-b \, \Plant=(a \,  \Generator) \, \Generator^{-1}-b \, \Plant$. This shows the necessity.

We next show the sufficiency. Lemma \ref{lem:::StablePlantSolvability} shows that there exists $\Controller_i$ that robustly regulates $b\, \Plant$. Now $\Controller_r=\Controller\, (1+\Controller_i)$ solves the robust regulation problem by Lemma \ref{lemma:product}.

\end{proof}


\begin{corollary}\label{cor:::Solvability}
Let $\Controller=\frac{b}{a}$ be a stabilizing controller of $\Plant$  such that $a,b\in\stable$ satisfy \eqref{eqn:::ab}. The robust regulation problem is solvable if and only if there exist $\alpha, \, \beta, \, q_1, \, q_2\in\stable$ such that:
\begin{align}\label{eqn:::SolvabilityCondition}
\alpha \, \Generator^{-1}-\beta \, (b+q_1 \, a^2+q_2 \, b^2) \,  \Plant=1.
\end{align}
\end{corollary}
\begin{proof}
The result follows from Theorem \ref{thm:::Solvability} and the parametrization \eqref{eqn:::AllStabilizingControllers} of stabilizing controllers. 
\end{proof}

\begin{remark}\label{rem:::RobustlyRegulatingController}
{\em 
If \eqref{eqn:::SolvabilityCondition} holds, then a stabilizing controller that satisfies the condition of Theorem~\ref{thm:::Solvability} is given by $\Controller=\frac{b+q_1 \, a^2+q_2 \, b^2}{a+q_1 \, \Plant \, a^2+q_2 \, \Plant \, b^2}$. The controller $\Controller_i$ in \eqref{eqn:::RobustlyRegulatingController} is to be designed so that it robustly regulates the stable plant 
$(b+q_1 \,  a^2+q_2 \, b^2) \, \Plant.$
Following the proof of Lemma \ref{lem:::StablePlantSolvability}, one such controller is $\Controller_i=\frac{\beta}{\alpha}\, \Generator$.
}
\end{remark}

For the rest of the section, we consider a generator $\Generator$ which admits a weakly coprime factorization. The next theorem is a simplification of Theorem~\ref{thm:::Solvability} with such a generator.

\begin{theorem}\label{thm:::CoprimeGeneratorSolvability}
If $\Generator=\frac{\gamma}{\generator}$ is a weakly coprime factorization, then the robust regulation problem is solvable if and only if the plant $\Plant$ is stabilizable and if there exist $\alpha, \, \beta\in\stable$ such that $\alpha \, \generator-\beta \, \Plant=1$.
\end{theorem}
\begin{proof}
We may assume that $\Plant$ is stabilizable. Let $\Controller$ be a stabilizing controller, i.e. there exist $a, \, b\in\stable$ such that $\Controller=\frac{b}{a}$ and \eqref{eqn:::ab} holds. 

In order to show the necessity, let us assume that $\Controller$ is robustly regulating. By Theorem~\ref{thm:::IMP_WeaklyCoprime}, there exist $\alpha_0, \, \beta_0\in\stable$ such that $\generator \, (\alpha_0+\beta_0 \, \Controller)=1$. By \eqref{eqn:::ab}, we have
\begin{align*}
1 & =\alpha_0 \, \generator+\generator \,  \beta_0 \, \Controller=\alpha_0\, \generator+\generator\, \beta_0\, \Controller\, (a-b\, \Plant) \\
& =(\alpha_0+\beta_0 \, b)\, \generator-(\beta_0\, \generator \, \Controller\, b) \, \Plant.
\end{align*}
Since $\alpha_0 + \beta_0 \, b\in \stable$ and $(\beta_0 \, \generator \, \Controller) \, b=(1-\generator \, \alpha_0) \, b \in \stable$, the necessity follows.

Let us now show the sufficiency.  Substituting
\begin{align*}
q & :=\beta \, a=\beta \, a \, (a-b \, \Plant)=\beta \, a^2-(\beta \, \Plant) \, a \, b \\
& =\beta \, a^2-(\beta \, \Plant) \, a \, b \, (a-b \, \Plant) \\
& = (1-p \, b) \, \beta \, a^2+(\beta \, \Plant) \, (a \, \Plant) \, b^2
\end{align*}
to (\ref{eqn:::AllStabilizingControllers}), where $\beta \, p=\alpha \, \theta+1 \in \stable$, and using the identities $\alpha \, \generator-\beta \, \Plant=1$ and $a-\Plant\, b=1$, we obtain the stabilizing controller
\begin{align*}
\Controller(\beta \, a) &=\Controller \, \frac{a \, (b+\beta \, a)}{b \, (a+\Plant  \, (\beta \, a))}=c\, \left(1+\frac{\beta \, a}{b \, (a+\Plant  \, (\beta \, a))} \right) \\ & =c \, \left( 1+\frac{\beta }{b \, (1+p \, \beta)} \right)=c \, \left( 1+\frac{\beta}{b \, \alpha \, \generator} \right)
\end{align*}
of $\Plant$ by Theorem~\ref{thm:::Stability}. Finally, we observe that the fractional representation $\Controller(\beta \, a)=\frac{\beta+\alpha\, \generator \, b}{\alpha\, \generator \, a}$ satisfies
\begin{align*}
\left\{\begin{array}{l}
\alpha\, \generator \, a-(\beta+\alpha\, \generator \, b)\;p=\alpha\, \generator \, (a-b\Plant) - \beta\,\Plant =1,\vspace{1mm}\\
\alpha\, \generator \, a\, \Plant\in\stable,
\end{array}\right.
\end{align*}
i.e. it satisfies (\ref{eqn:::ab}), and
\begin{align*}
(\alpha \, a\, \gamma) \, \Generator^{-1}-(\beta+\alpha\, \generator \, b)\, \Plant=1,
\end{align*}
so the claim follows by Theorem~\ref{thm:::Solvability}.
\end{proof}

By using Theorem~\ref{thm:::CoprimeGeneratorSolvability}, we are able to state the second main result of this section: a parametrization of all the robustly regulating controllers. The next theorem leading to parametrization of all robustly regulating controllers was given in \cite{Vidyasagar} for finite-dimensional systems. The actual parametrization will be given by Corollary~\ref{cor:::AllRobustlyRegulatingControllers}.

\begin{theorem}\label{thm:::AllRobustlyRegulatingControllers}
Assume that $\Generator=\frac{\gamma}{\generator}$ is a weakly coprime factorization. 
If the robust regulation problem is solvable,  then a controller $\Controller$ is robustly regulating if and only if it is of the form $\Controller=\frac{\Controller_0}{\generator}$, where $\Controller_0$ is a stabilizing controller of $\Plant_0:=\frac{\Plant}{\generator}$.
\end{theorem}
\begin{proof}
Assume that the robust regulation problem is solvable. We first show that if $\Controller_0$ stabilizes $\Plant_0$, then $\Controller$ is robustly regulating. Since we assume that $\Controller_0$ stabilizes $\Plant_0$, Theorem~\ref{thm:::Stability} implies that there exist stable elements $0\neq a,b\in\stable$ satisfying
\begin{align}\label{eqn:::ab2}
\left\{ \begin{array}{l}
a-\Plant_0 \,b=1, \vspace{1mm}\\
\Plant_0 \, a\in\stable,
\end{array}\right.
\end{align}
and $\Controller_0=\frac{b}{a}$. By \eqref{eqn:::ab2}, we see that:
\begin{align}\label{eqn:::Stability_help1}
\dfrac{1}{1-\Plant \, \Controller} =\dfrac{1}{1-\Plant_0 \, \Controller_0}=a\in\stable.
\end{align}
By the assumption that $\Controller_0$ stabilizes $\Plant_0$,
\begin{align}\label{eqn:::Stability_help2}
\dfrac{\Plant}{1-\Plant \, \Controller}  =
\dfrac{\generator \, \Plant_0}{1-\Plant_0 \, \Controller_0}\in\stable, \\
\label{eqn:::Regulation_help1}
\dfrac{1}{\generator}\dfrac{\Plant}{1-\Plant \, \Controller} =
\dfrac{\Plant_0}{1-\Plant_0\, \Controller_0}\in\stable.
\end{align}
Since the robust regulation problem is solvable, Theorem~\ref{thm:::CoprimeGeneratorSolvability} implies that there exist $\alpha,\beta\in\stable$ such that $\alpha \, \generator-\beta \,\Plant=1$. By \eqref{eqn:::ab2}, we have
\begin{align}\label{eqn:::Regulation_help2}
\dfrac{1}{\generator}\dfrac{1}{1-\Plant \, \Controller}=\dfrac{a}{\generator} \, (\alpha \, \generator-\beta \,\Plant)=a \, \alpha -(a \,\Plant_0) \, \beta \in\stable,\\
\label{eqn:::Stability_help3}
\dfrac{\Controller}{1-\Plant \, \Controller} =\dfrac{b}{\generator}=\dfrac{b}{\generator} \, (\alpha \, \generator-\beta \, \Plant)=b \, \alpha -(b \,\Plant_0) \, \beta\in\stable.
\end{align}
The controller $\Controller$ is stabilizing by \eqref{eqn:::Stability_help1}, \eqref{eqn:::Stability_help2} and \eqref{eqn:::Stability_help3}. It is regulating by \eqref{eqn:::Regulation_help1} and \eqref{eqn:::Regulation_help2}. The controller is robustly regulating since regulation implies robust regulation in the SISO case.

Next, we show that a robustly regulating controller has the form $\Controller=\frac{\Controller_0}{\generator}$ where $\Controller_0$ stabilizes $\Plant_0$. By Theorem~\ref{thm:::Stability}, $\Controller=\frac{b}{a}$, where $0\neq a\in\stable$ and $b\in\stable$, satisfy \eqref{eqn:::ab}. Since $\Controller$ is regulating for the signal generator $\generator^{-1}$ by Theorem~\ref{thm:::IMP_WeaklyCoprime} and \eqref{eqn:::ab} holds, we have:
\begin{align*}
\left\{ \begin{array}{l}
a-(\generator \, b) \, \Plant_0=a-\Plant \, b=1, \vspace{1mm}\\
\Plant_0 \, a=\dfrac{1}{\generator} \, \dfrac{\Plant}{1-\Plant \, \Controller}\in\stable.
\end{array}\right.
\end{align*}
This completes the proof since $\Controller_0=\generator \, \Controller$ stabilizes $\Plant_0$ by Theorem~\ref{thm:::Stability}.
\end{proof}

\begin{corollary}\label{cor:::AllRobustlyRegulatingControllers}
Let $\Controller$ be a robustly regulating controller. If $\Generator=\frac{\gamma}{\generator}$ is a weakly coprime factorization, then
all robustly regulating controllers are given by
\begin{align}\label{eqn:::AllRobustlyRegulatingControllers}
\Controller(q_1,q_2)=\dfrac{b+q_1 \, a^2+q_2 \,b^2}{\generator \, a+q_1 \,a^2 \, \Plant+q_2 \, b^2 \, \Plant},
\end{align}
where $a:=(1-\Plant \, \Controller)^{-1}$, $b:=\generator \, \Controller\,(1-\Plant \, \Controller)^{-1}$, and $q_1, \, q_2 \in A$ are arbitrary elements such that: 
\begin{align*}
\generator \,a+q_1 \, a^2\, \Plant+q_2 \,b^2 \, \Plant\neq 0.
\end{align*} 
\end{corollary}
\begin{proof}
Consider the notations of Theorem \ref{thm:::AllRobustlyRegulatingControllers}. Now $a=(1-\Plant_0 \, \Controller_0)^{-1}$ and $b=\Controller_0 \, (1-\Plant_0 \, \Controller_0)^{-1}$ satisfy
\begin{align*}
\left\{ \begin{array}{l}
a-\Plant_0 \, b=1, \vspace{1mm}\\
\Plant_0 \, a = \dfrac{1}{\generator} \, \dfrac{\Plant}{1-\Plant \, \Controller}\in\stable,
\end{array}\right.
\end{align*}
so Theorem~\ref{thm:::Stability} shows that all the stabilizing controllers of $\Plant_0$ are of the form:
\begin{align*}
\Controller_0(q_1,q_2)=\dfrac{b+q_1\, a^2+q_2 \, b^2}{a+q_1\, a^2 \, \Plant_0+q_2 \, b^2 \, \Plant_0}.
\end{align*}
Theorem~\ref{thm:::CoprimeGeneratorSolvability} shows that we obtain the desired parametrization by multiplying the above parametrization by $\generator^{-1}$.
\end{proof}




\section{Examples}\label{sec:::Examples}

In the first example, the plant does not possess a weakly coprime factorization. The second example shows that the results presented here extend the classical ones obtained in $\Hinfty$-framework. We will see that the signal generator need not possess a coprime factorization in order for the robust regulation problem to be solvable.

\begin{example}\label{exa:::Example1}
Recall \cite[Example 3.2]{Mori1999}, where $\stable:=\R[x^2,x^3]$ served as a discrete finite-time model of some high speed electronic circuits without unit delays. It was shown in \cite[Example 4]{Quadrat2003} that $\Plant:=\frac{x^3-1}{x^2-1}\in\fractions{\stable}$ does not admit a weakly coprime factorization over $\stable$ and that $\Controller:=\frac{x^2-1}{x^3+1}$ is a stabilizing controller. In addition, a fractional representation $\Controller=\frac{a}{b}$ that satisfies \eqref{eqn:::ab} is given by $a:=\frac{x^3+1}{2}$ and $b:=\frac{x^2-1}{2}$.

Let us consider robust regulation with the generator $\Generator:=\frac{1}{x^5-x^2+2}\in\fractions{\stable}$. If we choose $q_1=q_2=0$, $\alpha=\frac{1}{2}$, and $\beta=x^2$, then \eqref{eqn:::SolvabilityCondition} holds. The robust regulation problem is solvable by Corollary~\ref{cor:::Solvability}.

Let us now construct a robustly regulating controller. By Remark \ref{rem:::RobustlyRegulatingController}, a robustly regulating controller is given by \eqref{eqn:::RobustlyRegulatingController} if we can find a robustly regulating controller $\Controller_i$ for $b\, \Plant$. Following the proof of Lemma \ref{lem:::StablePlantSolvability} we find out that $\Controller_i=\frac{\beta}{\alpha\, \Generator^{-1}}$ robustly regulates $b\, \Plant$. The desired controller is:
\begin{align*}
\Controller_r=\Controller\,(1+\Controller_i)=\dfrac{(x^2-1)(x^5+x^2+2)}{(x^3+1)(x^5-x^2+2)}.
\end{align*}

We end this example by parameterizing all robustly regulating controllers. By Theorem \ref{thm:::AllRobustlyRegulatingControllers}, 
\begin{align*}
\Controller_0=(x^5-x^2+2) \, \Controller_r=\frac{(x^2-1)(x^5+x^2+2)}{x^3+1}
\end{align*}
stabilizes $\Plant_0=\frac{\Plant}{x^5-x^2+2}$. A fractional representation of $\Controller_0=\frac{a_0}{b_0}$ that satisfies
\begin{align*}
\left\{ \begin{array}{l}
a_0-\Plant_0 \, b_0=1, \vspace{1mm}\\
\Plant_0 \, a_0\in\stable,
\end{array}\right.
\end{align*}
is given by:
\begin{align*}
\left\{
\begin{array}{l}
a_0=\dfrac{1}{1-\Plant_0 \, \Controller_0}=\frac{x^3+1}{2}\in\stable, \vspace{1mm}\\
b_0=a_0 \, \Controller_0=\frac{x^2-1}{2}\in\stable.
\end{array}
\right.
\end{align*}
By Corollary \ref{cor:::AllRobustlyRegulatingControllers}, all the robustly regulating controllers of $\Plant$ are then given by:
\begin{align*}
\Controller(q_1,q_2)=\frac{b_0+q_1 \, a_0^2+q_2 \, b_0^2}{(x^5-x^2+2) \, a_0+q_1 \,  \Plant \, a_0^2+q_2 \,  \Plant \, b_0^2}.
\end{align*}\QEDexample
\end{example}

\begin{example}\label{exa:::WeaklyCoprimeGenerator}
Choose $\Plant:=1\in \Hinfty(\C_+)$ and let $\Generator\in \fractions{\Hinfty(\C_+)}$ be such that it does not possess a coprime factorization, e.g., see \cite{Logemann1987}. Any stabilizable plant in $\fractions{\Hinfty(\C_+)}$ possesses a coprime factorization \cite{Smith1989}, so $\Generator\,\Plant$ is not stabilizable. 
However, since $1=0\, \Generator^{-1}+\Plant$, there exists a robustly regulating controller by Lemma~ \ref{lem:::StablePlantSolvability}.

Let $\Generator=\frac{\gamma}{\generator}$ be an arbitrary fractional representation. We choose $\Controller=\frac{\generator-1}{\generator}$. It is easy to see that \eqref{eqn:::ab} holds with $a:=\generator$ and $b:=\generator-1$. The controller is stabilizing by Theorem~\ref{thm:::Stability}, and admits a coprime factorization $\Controller=\frac{\generator-1}{\generator}$. The controller is robustly regulating by Theorem~\ref{thm:Pcoprime}.
This shows that $\generator^{-1}$ is the internal model built into the controller. 

Above we have found a controller that solves the robust regulation problem. We know that $\Generator$ possess a weakly coprime factorization \cite{Smith1989}. Using it and Corollary \ref{cor:::AllRobustlyRegulatingControllers}, we can easily parametrize all the robustly regulating controllers.
\QEDexample
\end{example}


\section{Concluding Remarks}\label{sec:::Conclusions}

In this article, we have developed a frequency domain theory of robust regulation that uses no coprime factorizations for SISO systems. We were able to formulate the internal model principle and to give necessary and sufficient solvability conditions in a very general algebraic framework. In addition, a parametrization of all robustly regulating controllers was given provided that the signal generator possesses a weakly coprime factorization, but not necessarily a coprime factorization. Thus, the results of this article extend the classical ones using coprime factorization. If $\stable=\Hinfty$, this article fully characterizes the solvability and parametrizes all the robustly regulating controllers since any plant in $\fractions{\Hinfty}$ has a weak coprime factorization \cite{Smith1989}.

The advantage of the adopted 
approach is that the results of this paper extend the class of systems we can deal with, and gives a new formulation for some classical results using only general fractional representations. From the practical point of view, the usefulness of the results is a consequence of the difficulty to find coprime factorizations of the transfer functions of infinite-dimensional systems. Future work contains generalization of the results to the multi-input multi-output case.

\bibliographystyle{plain} 



\end{document}